\documentclass{amsart}

\usepackage{appendix}

\usepackage{color} % This package was added to write corrections in red

\usepackage{comment} % comment off sections of text

\makeatletter
\let\@wraptoccontribs\wraptoccontribs
\makeatother

\newcommand{\club}{\clubsuit} % THIS IS TO MARK COMMENTS, CHANGES

\newtheorem{theorem}{Theorem}[section]
\newtheorem{prop}[theorem]{Proposition}
\newtheorem{proposition}[theorem]{Proposition}
\newtheorem{lemma}[theorem]{Lemma}
\newtheorem{cor}[theorem]{Corollary}

\newtheorem{remark}[theorem]{Remark}

\newcommand{\dv}{{\rm dv}}

\newcommand{\vol}{{\rm vol}}

\newcommand{\EE}{{\mathcal{E}}}
\newcommand{\R}{{\mathbb{R}}}

\newcommand{\Diff}{{\rm Diff}}
\newcommand{\Prob}{{\rm{Prob}}}

\DeclareMathOperator{\Lie}{Lie}
\DeclareMathOperator{\Hom}{Hom}
\DeclareMathOperator{\Tr}{Tr}
\newcommand{\Sym}{\mathrm{Sym}}

\newcommand{\Pos}{\mathrm{Pos}}

\newcommand\norm[1]{\left\Vert {#1} \right\Vert}
\newcommand\isom{\simeq}

\DeclareMathOperator{\End}{End}
\DeclareMathOperator{\Met}{Met}
\newcommand\Metv{\Met_{\dv}}

\newcommand{\midmid}{\, \middle| \,}

\newcommand{\GL}{\mathrm{GL}}
\newcommand{\SL}{\mathrm{SL}}

\newcommand{\lsl}{\mathfrak{sl}}
\newcommand{\lgl}{\mathfrak{gl}}
\newcommand{\so}{\mathfrak{so}}
\newcommand{\gO}{\mathrm{O}}
\newcommand{\SO}{\mathrm{SO}}

\newcommand{\liea}{\mathfrak{a}}
\newcommand{\lieg}{\mathfrak{g}}

\newcommand{\FM}{\mathcal{F}(M)}
\DeclareMathOperator{\diag}{diag}
\DeclareMathOperator{\diam}{diam}
\DeclareMathOperator{\Stab}{Stab}

\newcommand\uxi{\underline{\xi}}
\newcommand\uet{\underline{\eta}}

\makeatletter
\@namedef{subjclassname@2010}{%
  \textup{2010} Mathematics Subject Classification}
\makeatother

\begin{document}

\title{Gaussian measures on the of space of Riemannian metrics}

\contrib[With an Appendix by]{Yaiza {Canzani*}, D. Jakobson and L. Silberman}

\author[B. Clarke]{Brian Clarke}
\address{Mathematisches Institut, 
Universit\"at M\"unster,  Germany}
\email{brian.clarke@uni-muenster.de}

\author[D. Jakobson]{Dmitry Jakobson}
\address{Department of Mathematics and Statistics, McGill University,
Montr\'eal, Canada.}
\email{jakobson@math.mcgill.ca}

\author[N. Kamran]{Niky Kamran}
\address{Department of Mathematics and Statistics, McGill University,
Montr\'eal, Canada.}
\email{nkamran@math.mcgill.ca}

\author[L. Silberman]{Lior Silberman}
\address{University of British Columbia, Department of Mathematics, Vancouver, Canada}
\email{lior@math.ubc.ca}

\author[J. Taylor]{Jonathan Taylor}
\address{Department of Statistics, Stanford University, Stanford, CA
94305} \email{jonathan.taylor@stanford.edu}

\date{\today}

\thanks{Y.C.\ was supported by Schulich Fellowship of McGill
  University (Canada).  B.C.~was supported by NSF grant DMS-0902674.
D.J.\ was supported by NSERC and FQRNT grants and Dawson Fellowship of McGill 
University (Canada). N.K. was supported by NSERC, FQRNT and James McGill Fellowship of McGill university. 
L.S. was supported by NSERC. \\
\indent   *YC: Department of Mathematics, Harvard University, USA. 
\emph{E-mail address}: canzani@math.harvard.edu}  



\begin{abstract}
We introduce Gaussian-type measures on the manifold of all metrics
with a fixed volume form on a compact Riemannian manifold of
dimension $\geq 3$.  For this random model we compute the characteristic
function for the $L^2$ (Ebin) distance to the reference metric.
In the Appendix, we study Lipschitz-type distance between Riemannian metrics,
and give applications to the diameter, eigenvalue and volume entropy
functionals.  \\
Nous introduisons des mesures de type gaussien sur la vari\'et\'e des m\'etriques riemanniennes a forme volume fix\'ee definies sur une vari\'et\'e compacte de dimension $\geq 3$. Pour ce mod\`ele al\'eatoire, nous calculons la fonction caract\'eristique de la distance $L^2$ (Ebin) \`a une m\'etrique de r\'ef\'erence donn\'ee. Dans un appendice, nous \'etudions une distance de type Lipschitz entre m\'etriques riemanniennes, et donnons des applications aux fonctionnelles associ\'ees ai diam\`etre, aux valeirs propres du laplacien et \`a l'entropie de volume.
\end{abstract}
\maketitle

\section{Introduction}

In this paper we construct Gaussian-type measures on the space of
Riemannian metrics on a fixed manifold and make some elementary observations
about them, leaving deeper results for further work.  We will being with
several motivations for our construction and with directions for further
work.

Let $(M,g)$ be a Riemannian manifold.  \emph{Quantum Chaos} is a general
term for the study of connections between the dynamics of the associated
geodesic flow on $TM$ (corresponding to the physics of a classical particle
moving freely on $M$) and the spectrum of the Laplace--Beltrami operator
on $L^2(M)$ (corresponding to the physics of a quantum particle moving
freely on $M$).  We note the conjectures of Bohigas, Giannoni and Schmidt
\cite{BGS} about asymptotic behaviour of level spacings between
Laplace eigenvalues for classically chaotic systems, and M. Berry's
random wave conjectures \cite{Berry} about asymptotic properties of
eigenfunctions.  These conjectures appear to be very difficult
to prove using standard semiclassical methods, and a natural idea is
instead to consider them \emph{on average} in some sense in the space of
metrics on $M$, or perhaps to use random methods to construct examples
or counter examples.

Another motivation for our construction is of developing geometric analysis
on (often infinite-dimensional) manifolds of metrics. 
Most important progress to date involved \emph{differentiation} on
manifolds of metrics, in particular the study of $L^2$ distance between
metrics and related questions \cite{Ebin,FG,Clarke}.  The next natural step
is to define \emph{integration} on manifolds of metrics, 
hence the need to define and study measures on those manifolds.
Related questions have been considered in \cite{Morgan}
(for manifolds of maps) and in \cite{CJW}.  

We now turn to our construction.  In the predecessor work \cite{CJW},
the authors took a fixed ``reference'' (or ``background'') metric $g^0$
on $M$, and then considered a random metric $g = e^{2\varphi}g^0$ in the
conformal class of $g^0$, where the (logarithm of) the conformal
factor $\varphi$ varied as a Gaussian random field on $M$, constructed
using the eigenfunctions of the Laplacian for the reference metric.
In the present paper, we work in a transverse direction: given the
reference metric $g^0$ we choose a random deformation $g$ among those
metrics having the same volume form as $g^0$.
Again we parametrize those metrics by exponentiating a Gaussian random
field on $M$.  Beyond describing the construction we limit out study to
the statistics of various distance functions on the space of metrics,
leaving deeper investigation for later papers.

\begin{remark}
It is possible to combine the two constructions, adding a conformal factor
to our deformation.  We mainly avoid doing this since the (completion of the)
space of all Riemannian metrics is singluar, unlike the case of a fixed volume
form.
\end{remark}

\begin{remark}
Our construction depends on a choice of a global orthonormal frame
in the tangent bundle (a global section of the frame bundle).
The existence of such a frame is known as \emph{parallelizability},
and is a topological property of $M$.  We do not believe this assumption
is essential; rather it simplifies the presentation here.  For example,
if one patches together deformations on parts of the manifold using
a partition of unity, the distance statistics would not be as nice.

With this choice in hand, our construction is equivariant under the
diffeomorphism group of the manifold: the pushforward of our probability
measure by a diffeomorphism is equal to the measure obtained by
pushing forward the reference metric and the frame.
\end{remark}

We give the construction in Section \ref{sec:Gaussian}.  It is based on
viewing the space of metrics with a given volume form as the space of
sections of a bundle over $M$ with fibers diffeomophic
to the symmetric space $S = \SL_n(\R)/\SO(n)$ ($n=\dim M$).
This symmetric space supports an invariant Riemannian metric
which can then be used to define an $L^2$ distance on the space of metrics,
which coincides with the distance arising from a Riemannian structure on this
(infinite-dimensional) space.  This distance is introduced in
Section \ref{sub:L2-distance} and is studied as a random variable in
Section \ref{sec:fixvolume:om2}, where tail estimates
are obtained in terms of geometric constants.

In the Appendix, similar computations are carried out for a Lispchitz-type
distance, also considered in \cite{BU}.  Those estimates are then 
applied to establish integrability and existence of exponential moments for 
the \emph{diameter}, \emph{Laplace eigenvalue} and \emph{volume entropy}
functionals of our random Riemannian metrics.

Initial directions for further work involves studying the nature of the
deformation we obtain (computation of the probability of the metric to lie
in a small ball around the reference metric, and the behaviour of the
isoperimetic constant under the deformation).  In a foundational direction,
we will address in a sequel questions about \emph{convergence}
and {\em tightness} (i.e.\ relative compactness in
the weak-* topology) of our families of measures.  

We expect that the Gaussian measure we have introduced in this paper
will have applications that extend significantly beyond the basic
questions considered here, in particular to the motivating problems
discussed above.

\subsection*{Acknowledgements}
The authors would like to thank R.~Adler, R.~McCann, A.~Nabutovsky,
I.~Polterovich, P~ Sarnak, N.~Sidorova, P.~Sosoe, J.~Toth, S.~Weinberger,
I.~Wigman and S.~Zelditch for stimulating discussions.  Specific thanks to
L.~Hillairet for finding an error in a previous version of the appendix.
The authors would like to thank for their hospitality the organizers of the
following conferences, where part of this research was conducted:
Spectral Geometry Workshop at Dartmouth (July 2010); Workshop on Manifolds of
Metrics and Probabilistic Methods in Geometry and Analysis at CRM,
Montreal (July 2012); Workshop on Metric Geometry, Geometric Topology
and Groups at BIRS (August 2013); Workshop on Infinite-Dimensional Geometry at 
Berkeley (December 2013).  In addition, D.J. would like to acknowledge
the hospitality of the Departments of Mathematics at Stanford University and
at UBC, where this research was partly conducted.

%%%%%%%%

\section{The space of metrics}\label{sec:space-of-metrics}

We fix once and for all a compact smooth manifold $M$ without boundary and
write $n$ for its dimension.  We also fix a smooth volume form $\dv$ on $M$.

We rely crucially on the symmetric space structure of the space
$P$ of positive-definite matrices of determinant $1$ and on the related
structure theory of $\SL_n(\R)$.  In the discussion below we state the
facts we use; proofs and further details may be found in the text
\cite{Terras}, which concentrates on this case, and in \cite{Helgason} which
develops the general theory of symmetric spaces associated to semisimple
Lie groups.

\subsection{The space of metrics}\label{sub:space-metrics}

We start by giving a coordinate-free description of the set of Riemannian
metrics with the volume form $\dv$ on $M$.  We then restrict to a class of
manifolds for which there is a coordinate system simplifying the description.

\begin{comment}
The first description will be in terms of reductions of the frame bundle
$\FM$, viewed as a principal $\GL_n(\R)$-bundle over $M$.
Recall \cite{Sternberg} that for given subgroup
$G \subset \GL_n(\R)$, a $G$-structure $B_{G}$ on $\FM$ is a reduction of
the bundle to that group, that is a $G$-invariant submanifold of
$\FM$ surjecting on $M$ such that for all $p \in B_{G}$ and
$g\in \GL_n(\R)$, we have $g \cdot p\in B_{G}$ if and only if $g\in G$.
In this language the set of all Riemannian metrics on $M$ is the set of
$\gO(n)$-structures on $\FM$.  Now let $B_{G}$ denote
the $\gO(n)$-structure on $M$ associated to $g_{0}$ and let $\SL_n(\R)$
act pointwise on the fibers of $\pi:B_{G}\to M$.  Note that the submanifold
of $\FM$ obtained by this construction corresponds to the
set of orthonormal frames for
the Riemannian metrics $g$ on $M$ having the same volume form as
$g_{0}$. 

We shall now give a second description which forms our main point of
view.
\end{comment}

Let $V$ be a finite-dimensional real vector space with dual space $V^*$,
and let $\Sym (V) = \left\{ g\in\Hom(V,V^*) \vert g^* = g\right\}$ be
the space of symmetric bilinear forms on $V$.  Among those we distinguish
$\Pos (V) = \left\{g\in \Sym (V) \vert \forall v\in
  V:g(v,v) > 0\right\}$, the space of positive-definite bilinear
forms on $V$.  Let $\SL(V)\subset \GL(V)$ denote the special
and general linear groups on $V$, and $\lsl (V)\subset \lgl
(V)$ their Lie algebras.  Then $\GL(V)$ acts on $\Pos(V)$ by
\begin{equation}\label{eq:cong-action}
h^{-1} \cdot g = h^*\circ g \circ h\,.
\end{equation}
It is well-known that this action is transitive;
the stabilizer of any $h\in\Pos(V)$ is a maximal compact subgroup
isomorphic to $\gO(n)$.  Moreover, the orbits of $\SL(V)$ are
precisely the level sets of the determinant function $g\mapsto
\det(g_0^{-1} g)$ where $g_0$ is a fixed isomorphism $V\to V^*$.  Each
level set is then of the form $\SL(V)/K_{g_0}$ where $K_{g_0} =
\Stab_{\SL(V)}(g_0) \isom \SO(n)$ and we give it the
$\SL(V)$-invariant Riemannian structure coming from the Killing form
of $\SL(V)$, making it into a simply connected Riemannian manifold of
non-positive curvature.

\begin{remark} Since $\Pos(V)$ is an open subset of the vector space $\Sym(V)$,
we may trivialize its tangent bundle by identifying each tangent space with
$\Sym(V)$.  The reader may then verify that with this identification, the
tangent space at $g$ to the $\SL(V)$-orbit of of $g$ is exactly
$\{ X \in \Sym(V) \vert \Tr(g^{-1}X) = 0 \}$.  Here we compose the linear maps
$X\in\Hom(V,V^*)$ and $g^{-1} \in\Hom(V^*,V)$ to obtain a map in $\End(V)$
which has a trace.  The reader may also verify that, since the congruence
action above is linear as an action on $\Sym(V)$, the derivative of
the action of $h^{-1}$ at $g$ is the map $X\mapsto h^{-1}Xh$ (composition
of linear maps).

Now the Riemannian structure on the orbit claimed above is
\begin{equation}\label{eq:symmspace}
\rho_g(X,X) = \Tr\left(g^{-1}Xg^{-1}X\right)\,,
\end{equation}
and it is an immediate
calculation that this is $\SL(V)$-equivariant: that
$\rho_{h.g}(h.X,h.X) = \rho_g(X,X)$, in other words that the metric
is $\SL(V)$-invariant.
\end{remark}

With the usual translation of notions from vector spaces to vector bundles,
we associate to the tangent bundle $TM$ the vector bundles $\Hom(TM,T^*M)$
and $\Sym(TM)$, the symmetric space-valued bundle $\Pos(TM)$, and the
group bundles $\GL(TM)$ and $\SL(TM)$.

By definition, a \emph{Riemannian metric} on $M$ is a smooth section of
$\Pos(M)$; we denote the space of sections by $\Met(M)$.  To such a
metric there is an associated Riemannian volume form, and we let
$\Metv(M)$ denote the space of metrics whose volume form is $\dv$.
Fixing a metric $g_0 \in \Metv(M)$, the above discussion identifies
$\Metv(M)$ with the space of sections of the bundle over $M$ whose fibers
are isomorphic to $\SL_n(\R)/\SO(n)$.  Moreover, the fibre at $x$ of
this bundle is equipped with a transitive isometric action of $\SL(T_x
M)$, where the metric is the one pulled back from the identification
with $S = \SL_n(\R)/\SO(n)$ (the pullback is well-defined since the
metric on $S$ is $\SL_n(\R)$-invariant).

\begin{remark}
It is a classical result of Ebin \cite{Ebin} that the diffeomorphism group
acts transitively on the space of smooth volume forms of total volume $1$,
and therefore that the foliation of $\Met(M)$ by the orbits of the
diffeomorphism group $\Diff(M)$ descends to a foliation of $\Metv(M)$ by
the group $\Diff_\dv(M)$ of volume-preserving diffeomorphisms.  It follows that
$\Met(M)/\Diff(M) \isom \Metv(M)/\Diff_\dv(M)$; we regard this space
as the \emph{space of geometries} on $M$.
\end{remark}

In local co-ordinates $(x^1,\ldots,x^n)$, the above construction reads
as follows. One takes the basis
$\left\{\frac{\partial}{\partial x^i}\right\}_{i=1}^{n}$ for $T_x M$
and its dual basis $\left\{dx^i\right\}_{i=1}^{n}$ for $T_x^* M$.
Then fibers of $\Sym(M)$ are represented by symmetric matrices,
fibers of of $\Pos(M)$ by positive-definite symmetric matrices.
The volume form associated to $g\in\Met(M)$ is then given by
$\left|\det(g_x)\right|^{1/2} dx^1 \wedge\cdots\wedge dx^n$.
$\Metv(M)$ is then the metrics $g$ such that $\det(g_x) = \det(g^0_x)$
for all $x\in M$, where $g^0$ is any metric with Riemannian volume
form $\dv$.
The group $\GL_n(\R)$ then acts on the fibres via congruence transformations
$h^{-1} \cdot g = h^t g h$, with the
stabilizer of $g_x$ being the orthogonal group $\gO_{g_x}(\R)\isom\gO(n)$.
Similarly, the group $\SL_n(\R)$ acts transitively on the subset of the
fibre with a given determinant, with point stabilizer
$\SO_{g_x}(\R) \isom \SO(n)$.

\subsection{Deforming a metric}
Fix $g^0 \in \Metv(M)$, and for $x\in M$ let $K_x \subset G_x = \SL(T_x M)$
be the orthogonal group of the positive-definite quadartic form $g^0_x$,
which is also the stabilizer of $g^0_x$ under the congruence action
\eqref{eq:cong-action}.  Fix a frame $f_x$ in $T_x M$, orthonormal
with respect to the inner product defined by $g^0_x$, and let
$A_x \subset G_x$ be the subgroup of matrices which are diagonal with positive
entries in the basis $f_x$.  As noted above we can identify the set of
positive-definite quadratic forms on $T_x M$ with the same determinant as
$g^0$ with the symmetric space $G_x / K_x$.

\begin{remark} We warn the reader that we use the usual letter $G$ to
denote a semisimple Lie group and the letter $g$ to denote a Riemannian metric.
As such we don't have $g_x \in G_x$, and rather use $h_x$ to denote an
arbitrary element of $G_x$.
\end{remark}

Recall now the \emph{Cartan} decomposition
\begin{equation}\label{eq:Cartan}
G_x = K_x A_x K_x
\end{equation}
(see for example \cite{Terras}).  This states that every
$h_x \in G_x$ can be written in the form
$h_x = k_{1,x} a_x k_{2,x}$ with $k_{i,x}\in K_x$ and $a_x\in A_x$,
with $a_x$ being unique up to the action of the \emph{Weyl group}
$N_{G_x}(A_x) / Z_{G_x}(A_x)$, a group isomorphic to $S_n$ acting by
permutation of the coordinates with respect to the basis $f_x$.  Given
$a_x$, the two elements $k_{i,x}\in K_x$ are unique up to the fact that
$Z_{K_x}(a_x)$ may not be trivial (generically this centralizer is equal
to $Z_{K_x}(A_x)$, which is either trivial or $\{ \pm 1 \}$ depending on
whether $n$ is odd or even).

Recalling that $k_{2,x}\in K_x$ stabilizes $g^0_x$, it follows
that for $h_x\in G_x$ decomposed as above we have
$$ h_x \cdot g^0_x = (k_{1,x} a_x)\cdot g^0_x.$$
Since $G_x$ acts transitively on the level set, it follows that every
$g^1_x$ with the same determinant $g^0_x$ is of this form, and moreover that
in that form the $a_x$ is unique up to the action of $S_n$ on $A_x$.

Our goal is to randomly deform $g^0$ by choosing elements $k_x$ and $a_x$ 
for every $x\in M$.  We shall discuss the ``random'' aspect of the construction
in the next section, and concentrate at the moment on the topological
issues involved in making such constructions well-defined.

Given the orthonormal frame $f_x$, we can identify $A_x$ with the space
of positive diagonal matrices of determinant $1$.  Further, using the
exponential map we may identify this group with its Lie algebra
$\liea \isom \R^{n-1}$ of diagonal matrices of trace zero.
We will therefore specify $a_x$ by choosing such a matrix at each $x$,
that is by choosing a function $H\colon M\to \liea$.

While this clearly works locally, making a global
identification requires a choice of frame $f_x$ at
every $x\in M$, that is an everywhere non-zero section of the frame bundle
of $M$ or equivalently a trivialization of the tangent bundle of $M$,
something which is not possible in general.
For simplicity we have decided to only discuss here the
case of manifolds where such sections exist, and defer more general
constructions to future papers.

\begin{remark}
We required the existence of a smooth $g^0$-orthonormal frame.  However,
this is equivalent to the topological condition (``parallelizability'')
of the existence of a smooth but not necessarily orthonormal frame.
To see this note that starting with any non-zero smooth section of the frame
bundle, applying pointwise the Gram--Schmidt procedure with
respect to the metric $g^0$ is a smooth operation and
will produce a smooth orthonormal frame.
\end{remark}

We survey here some facts about parallelizable manifolds, mainly to note
that this class is rich enough to make our construction interesting.
First, a parallelizable manifold is clearly orientable. Second, 
a necessary condition for parallelizability is the vanishing of the second
Stiefel--Whitney class of the tangent bundle, which for orientable manifolds
is equivalent to $M$ being a \emph{spin manifold}.
Examples of parallelizable manifolds include all $3$-manifolds,
all Lie groups, the frame bundle of any manifold and 
the spheres $S^n$ with $n\in\{1,3,7\}$.

\subsection{The $L^2$ metric}\label{sub:L2-distance}
Once the volume form is fixed, the action of $\SL(T_x M)$ on the stalk
of $\Metv(M)$ at $x$ identifies it with the symmetric space
$S = \SL_n(\R)/\SO(n)$.  As noted above this space supports an
$\SL_n(\R)$-invariant Riemannian metric of non-positive curvature.
Denote its distance function $d_S$; we then write $d_x$ for the
well-defined metric on the stalk at $x$ of $\Metv(M)$.
Integrating this over $M$ then gives a metric (to be denoted $\Omega_2$)
on $\Metv(M)$: given two Riemannian metrics $g^0, g^1 \in \Metv(M)$ on $M$
with the same Riemannian volume form $\dv$, we set

$$\Omega_2^2(g^0, g^1) = \int_M d_x^2(g^0_x, g^1_x) \dv(x).$$

For a different point of view on this metric, recall that $d_S$ is the
distance function associated to the Riemannian metric \eqref{eq:symmspace}.
Fixing $V=\R^n$ with its standard metric and frame, we write $G=\SL_n(\R)$,
$K = \SO(n)$ so that $S=G/K$.
In this setting one can find directly the geodesics connecting
the standard metric to any metric which is diagonal in the standard basis.
Using $G$-equivariance and the Cartan decomposition \eqref{eq:Cartan},
the upshot is the following (for details see \cite{Terras}):
let $hK, h'K \in S=G/K$ correspond to two metrics of equal
determinant.  Then $K h^{-1} h' K$ is a well-defined element of
$K\backslash G / K \isom A/S_n$, where $A$ is the group of diagonal matrices
of determinant $1$ and positive entries.  Let $a\in A$ be a representative
for $K h^{-1} h' K$.  We then say that $g$ and $h$ are in
\emph{relative position} $a$.  Writing $\log a$ for the vector of $n$
logarithms of the diagonal entries of $a$ (note that the entries of $\log a$
sum to $1$, since $\det a=1$), 
it turns out that 
and $d_S(hK, hK) = \norm{\log a}$, where
and $\norm{\cdot}$ is the usual $\ell^2$ norm.

%%%%%%%%%%%%%%

\section{Gaussian measures on the space of metrics}\label{sec:Gaussian}

We next turn to the question of actually constructing our Gaussian measures.
For a general reference on Gaussian random variables see \cite{Bo}.
In view of the decomposition considered in Section \ref{sub:space-metrics}, 
it is natural to split the construction into diagonal and orthogonal parts.

Let $g^0$ be our reference metric.  Every other metric of $\Metv$ is of the
form $g^1_x = k_x a_x \cdot g^0_x$ where $k, a$ are smooth functions on $M$
such that $k_x\in K_x$ and $a_x\in A_x$.  In Sections \ref{subsec:diag} and
\ref{subsec:angular} we describe random constructions of $a_x$ and $k_x$
respectively.

It is not hard to verify that $\bigcup_x K_x,\bigcup_x A_x,\bigcup_x G_x$
are subbundles of the Lie-group bundle $\GL(TM)$, and that their Lie
algebras therefore furnish subbundles of the Lie algebra bundle 
$\lgl(TM) \isom \End(TM)$.  Specifically, $\Lie(G_x)$ consists of the
endomorphisms of $T_x M$ of trace zero, $\Lie(K_x)$ consists of the
endomorphisms which are skew-symmetric in the frame $f_x$, and $\Lie(A_x)$
consists of those which are diagonal of trace zero in the frame.

For the constructions below we fix a complete orthonormal basis
$\left\{\psi_j\right\}_{j=0}^{\infty} \subset L^2(M)$ 
such that $\Delta_{g^0}\psi_j+\lambda_j\psi_j=0$, with $\lambda_j$
being a non-decreasing ordering of the spectrum of the Laplace
operator $\Delta_{g^0}$.  Our constructions are in fact independent of the
choice of basis of each eigenspace, but it is more convenient to make an
explcit choice.

\subsection{The radial part}\label{subsec:diag}

We begin by defining a measure on the space of smooth functions $x\mapsto H_x$
such that $H_x \in \Lie(A_x)$ (sections of the bundle $\bigcup_x \Lie(A_x)$).
We follow the recipe of \cite{Morgan}:
choose \emph{decay coefficients} $\beta_j = F(\lambda_j)$ where $F(t)$
is an eventually monotonically decreasing
function of $t$ and $F(t)\to 0$ as $t\to\infty$.  Then set

\begin{equation}\label{rand:radial}
H_x = \sum\limits_{j=1}^{\infty}\pi_n(\uxi_j) \beta_j\psi_j(x),
\end{equation}
where $\uxi_j$ are i.i.d.~standard Gaussians in $\R^{n}$, and
$\pi_n\colon\R^n \to \R^n$ is the orthogonal projection on the
hyperplane $\sum_{i=1}^{n} x_i = 0$.

Finally, set $$ a_x = \exp(H_x)$$
where $\exp$ is the exponential map to $A_x$ from its Lie algebra.

The smoothness of $H$ defined by \eqref{rand:radial} is given by
\cite[Theorem 6.3]{Morgan}.  The following two propositions apply
whenver $\uxi_j$ in \eqref{rand:radial} denotes a
$d$-dimensional standard Gaussian, while $M$ has dimension $n$.

\begin{prop}\label{prop:morgan}
If $\beta_j=O(j^{-r})$ where $r>(q+\alpha)/n+1/2$, then $H$ defined
by \eqref{rand:radial} converges a.s. in
$C^{q,\alpha}(M,\R^d)$.
\end{prop}
We remark that the exponents in Proposition \ref{prop:morgan} are
independent of $d$ (the dimension of the ``target'' space).
Now Weyl's law for $M$ \cite{MP,Av} states that $\lambda_j$ grows roughly
as $j^{2/n}$.  It follows that
\begin{prop}\label{diagonal:smooth}
If $\beta_j=O(\lambda_j^{-s})$ where $s>q/2+n/4$, then $H$ defined
by \eqref{rand:radial} converges a.s. in $C^q(M,\R^d)$.
\end{prop}

%%%%%%%%

\subsection{The angular part}\label{subsec:angular}

In this paper we study invariants of $g^1$ that can be bound only using
$a$, so that our later calculations will only depend on the marginal
distribution of $a$.  Thus, as long as the choices of $k$ and $a$ are
independent, the choice of $k$ has no effect.
In future work we plan to ask more detailed questions where this choice
will become relevant.  For example,
determining the curvature of $g^1$ following the ideas of \cite{CJW}
requires differentiating $g^1_x$ with respect to $x$ and this immediately
involves the choice of $k_x$.  We thus propose the following specific choice,
again using the recipe of Equation \eqref{rand:radial}.  We set
$$k_x = \exp_x(u_x)$$
where $u_x$ is the Gaussian vector
\begin{equation}\label{rand:angular}
u_x = \sum_{j=1}^{\infty} \uet_j \delta_j \psi_j(x).
\end{equation}
Here $\uet_j\in\so_n$ are i.i.d.~standard Gaussian anti-symmetric matrices
(i.e.~each $\uet_j$ is given by $d_n = n(n-1)/2$ i.i.d.\ standard Gaussian
variables corresponding to the upper-triangular part of $\uet_j$),
and $\delta_j = F_2(\lambda_j)$ are decay factors, given as functions of
the corresponding eigenvalues.

Proposition \ref{prop:morgan} above applies again to give the smoothness
properties of our random sections.  In particular, since the exponents in
Proposition \ref{prop:morgan} are independent of $d_n$, substituting
into Weyl's law we get a straightforward analogue of Proposition 
\ref{diagonal:smooth} for the expression \eqref{rand:angular}.  

%\begin{prop}\label{orthogonal:smooth}
%If $\delta_j=O(\lambda_j^{-s})$ where $s>q/2+n/4$, then $u$ defined
%by \eqref{rand:angular} converges a.s. in $C^q(M,\R^n)$.
%\end{prop}

\subsection{Remarks on the construction}

For readers who may not wish to refer to a textbook such as \cite{Bo},
we briefly recall that a random vector such as $H_x$ is $\emph{Gaussian}$
if its finite-dimensional marginals are, which in our case roughly means
(though we want more) that for every $k$ points $x_1,\ldots,x_k\in M$,
the joint distribution of the finite-dimensional vector
$(H_{x_1},\ldots,H_{x_k})$ is Gaussian.

Our Gaussian vectors are balanced (their expectation is zero) and they
are therefore determined by their covariance function (roughly, the
function on $M\times M$ given by the expectation of $H_{x_1}\otimes H_{x_2}$.

\begin{remark}
For the convenience of the reader who prefers Gaussian variables to be
defined by their covariance function, we note here the covariance functions
relevant to our case.

Let $\lieg_x = \lsl(T_xM)$ denote the Lie algebra of $\SL(T_x M)$.  
As noted above our Gaussian measure is defined on appropriate spaces of
sections of subbundles of the bundle $\bigcup_x \lieg_x$.
With sufficient continuity it is enough to consider the covariance operator
evaluated on linear functionals of the form $X\mapsto \alpha_x(X(x))$,
where $X$ is a section of the bundle and $\alpha_x \in \lieg_x^*$.  

Our Gaussian measure for the diagonal part then has the covariance functions
\begin{equation}\label{eq:covariance}
R((x,k),(x',k'))=\delta_{kk'}\sum_j \beta_j^2 \psi_j(x)\psi_j(x')\,,  
\end{equation}
where $k$ is an index for the diagonal entries of a matrix in $\lieg_x$, 
diagonal with respect to our fixed frame and $(x,k)$ therefore denote the
functional mapping the section $H_x$ to the $k$th entry of the diagonal matrix at $x$.
.  The angular part has a similar
covariance function.

For standard choices of $\beta_k$, we note that the covariance function for
analogously-defined scalar fields would be well-known spectral invariants:
we'd have 
\begin{equation}\label{covar:diagonal}
r(x,y)=
\begin{cases}
Z(x,y,2s):=\sum_{k=1}^\infty\frac{\psi_k(x)\psi_k(y)}{\lambda_k^{2s}},
\qquad\beta_k=\lambda_k^{-s};\\
e^*(x,y,2t):=\sum_{k=1}^\infty\frac{\psi_k(x)\psi_k(y)}{e^{2t\lambda_k}},
\qquad\beta_k=e^{-t\lambda_k}.
\end{cases}
\end{equation}
Here $Z(x,y,2s)$ is known as the \emph{spectral zeta function}
of $\Delta_0$ (see e.g. \cite{MP}), while $e^*(x,y,2t)$ is the corresponding \emph{heat kernel} 
(see e.g. \cite{BGV} or \cite[Ch. 6]{Chavel}),  
both taken \emph{without the constant term} that would correspond to the
constant eigenfunction $\psi_0$ with eigenvalue zero.
\end{remark}

\begin{remark}
When taking $\beta_j = \lambda_j^{-s}$, the parameter $s$
determines the a.s.~Sobolev regularity of the random metric $g$
via Propositions \ref{prop:morgan} and \ref{diagonal:smooth}.
If the metric $g^0$ is
real-analytic, then letting $\beta_k=e^{-t\lambda_k}$ makes the
random metric $g$ real-analytic as well, with the parameter $t$
related to the a.s.~radius of analyticity (the exponent in rate of
decay of Fourier coefficients).
\end{remark}

\begin{remark}
A similar construction applies to the space of \emph{all}
Riemannian metrics on $M$ (without necessarily fixing the
volume form).  We now work in the symmetric space $\GL(T_x M)/\gO(g^0_x)$.
The only change is that in Equation \eqref{rand:radial} one lets $A_j$
be standard vector-valued Gaussians without the projection.

There is a Riemannian structure and an $L^2$ metric (due to Ebin) defined on
the space of all metrics.  A detailed study of the metric properties of this
space was undertaken in \cite{Clarke}.
\end{remark}

\begin{comment}
\begin{remark}
Gaussian measures on the space of metrics considered here induce Gaussian measures on the 
identity component of the group of diffeomorphisms of $M$.  DISCUSS THIS FURTHER. 
\end{remark}
\end{comment}
%%%%%%
%%%%%%%%%%

%%%%%%%%%%

\section{$\Omega_2$ as a random variable}\label{sec:fixvolume:om2}

In this section we study the statistics of $\Omega_2^2$.

\subsection{The distribution function}

We recall one definition of the (fiber-wise) distance $d_x$ introduced
in Section \ref{sub:L2-distance}. For this choose a a basis for $T_x
M$ orthonormal with respect to $g^0(x)$ (in this basis the reference
metric $g^0_x$ is represented by the identity matrix).  If the translation
from $g^0_x$ to $g^1_x$ is given by the element $k_x a_x \in G_x$ with
$a_x$ diagonal in the chosen basis, $k_x$ orthogonal, then the metric
$g^1_x$ is represented by the symmetric positive-definite matrix
$k_x a_x^2 k_x^{-1}$.  Writing $e^{b_i(x)}$ for the diagonal entries
of $a_x$, we have
$$ d_x^2(g^0_x,g^1_x) = \sum_{i=1}^{n} b_i(x)^2. $$

Accordingly,
\begin{equation}\label{Omega2:formula}
\Omega_2^2(g^0,g^1) = \int_M \left(\sum_{i=1}^n b_i(x)^2\right) \dv(x).
\end{equation}

In our random model, the vector-valued function $b(x)$ is a Gaussian
random field, chosen according to Equation \eqref{rand:radial}, where
here we choose $\pi_n$ to be the orthogonal projection.  In other
words $b(x)$ is defined by projecting an isotropic Gaussian in $\R^n$
orthogonally to the hyperplane $\sum_i b_i(x) = 0$. Integrating over
$x$, we find that the distribution of $\Omega_2^2$ is given by:
$$
\Omega_2^2 \overset{D}{=} \sum_j \beta_j^2 \sum_{i=1}^{n-1}
W_{i,j}
$$
where the $W_{i,j}$ are independent random variables with $\chi^2$
distribution.  We can rewrite this as
$$
\Omega_2^2 \overset{D}{=} \sum_j \beta_j^2 V_j
$$
with i.i.d.~$V_j \sim \chi^2_{n-1}$ ($\chi^2$ distribution with $n-1$
degrees of freedom).

Recall that the \emph{moment generating function} of the random
variable $X$ is the function $M_X(t) = \EE\left(\exp(tX)\right)$.
These can be used, for example, to estimate the probability of large
deviations of the variable $X$.  Having represented $\Omega_2^2$ as the
sum of independent variables with known distribution, we can now explicitly
compute its moment generating function as the product
$$
\begin{aligned}
M_{\Omega^2_2}(t) &= \prod_j \prod_{i=1}^{n-1} M_{\chi^2_1}(t\beta_j^2)
                      = \prod_j \prod_{i=1}^{n-1} (1-2t \beta_j^2)^{-1/2}\\
               \; &= \prod_j(1-2t \beta_j^2)^{-(n-1)/2}
\end{aligned}
$$

The following result is proved similarly. 
\begin{prop}\label{char:fxn}
The characteristic function $E(\exp(it \Omega^2_2))$ can be computed explicitly as
$$
\prod_j \prod_{k=1}^{n-1} (1-2it \beta_j^2)^{-1/2}
=\prod_j(1-2it \beta_j^2)^{-(n-1)/2}.  
$$
\end{prop}

\subsection{Tail estimates for $\Omega_2^2$}\label{sec:tail:l2}

Here we apply \cite[Lemma 1, (4.1)]{LM} to estimate the probability of
the following events:
\begin{equation}\label{pro:l2:large}
\Prob\{\Omega_2^2>R^2\},\qquad R\to\infty.
\end{equation}

We let $W=\sum_i a_i Z_i^2$ with $Z_i$ i.i.d.~standard Gaussians,
and for $(n-1)(j-1)+1\leq i\leq (n-1)j$, we have $a_i=\beta_j^2$
(i.e.~each $\beta_j^2$ is repeated $(n-1)$ times). We let 
$\norm{a}_\infty=\sup_j a_j$.  Assume from now on that $\beta_j=F(\lambda_j)$ 
is a {\em monotone
decreasing} function; then $\norm{a}_\infty=a_1=\beta_1^2$.

It is shown in 
\cite[Lemma 1, (4.1)]{LM} that for 
$W_k=\sum_{i=1}^{k(n-1)}a_i Z_i^2$, we have 
$$
\Prob\{W_k\geq\sum_{i=1}^{k(n-1)}a_i+2\left(\sum_{i=1}^{k(n-1)}a_i^2\right)^{1/2}\sqrt{y} + 
2\norm{a}_\infty y\} 
\leq e^{-y}.  
$$
Letting $k\to\infty$, we get the following quantities: 
\begin{itemize}
\item[] $W:=\lim_{k\to\infty}W_k=\Omega_2^2$; 
\item[] $A^2=\sum_{i=1}^\infty a_i=(n-1)\sum_{j=1}^\infty\beta_j^2$; 
\item[] $B^4=\sum_{i=1}^\infty a_i^2=(n-1)\sum_{j=1}^\infty\beta_j^4$; 
\item[] $\norm{a}_\infty=a_1=\beta_1^2$.  
\end{itemize}
With $x^2$ instead of $y$, we get:
$$
\Prob\{W\geq A^2 + 2B^2 x +2\norm{a}_\infty^2 x^2\} 
\leq e^{-x^2}.  
$$
Solving
$$
R^2=2||a||_\infty^2 x^2+2B^2 x+A^2.  
$$
for $x$ gives (for $R\geq A$) the following root:
\begin{equation}\label{x-of-R}
x(R)=\frac{-B^2+\sqrt{B^4+2(R^2-A^2)||a||_\infty^2}}{2||a||_\infty^2}.  
\end{equation}
We conclude that
$$
\Prob\{\Omega_2\geq R\} \leq e^{-(x(R))^2}, 
$$ 
where $x(R)$ is given by \eqref{x-of-R}.  

It is easy to show that there exists a constant $C=C(A,B,||a||_\infty)$ such that 
for $R\geq A$, we have 
$$
x(R)^2\geq\frac{R^2}{2||a||_\infty^2}-CR=\frac{R^2}{2\beta_1^2}-CR.  
$$

We also notice that 
$$
\Prob\{\Omega_2\geq R\}\geq\Prob\{\beta_1^2Z_1^2\geq R^2\}=
\Psi\left(\frac{R}{\beta_1}\right) 
\geq C\frac{\beta_1e^{-R^2/(2\beta_1^2)}}{R}, 
$$
provided $R\geq \beta_1$.

To summarize: 
\begin{theorem}\label{tail:Omega_2}
For $R\geq A$, we have 
$$
\frac{C\beta_1}{R}\exp\left(\frac{-R^2}{2\beta_1^2}\right)
\leq \Prob\{\Omega_2\geq R\}
\leq \exp\left(\frac{-R^2}{2\beta_1^2}+CR\right).  
$$
\end{theorem}

%%%%%%%%%%%%%%%%%%%%%%%%%%%%

\appendix \setcounter{section}{-1}
\renewcommand{\thesection}{A}

\section*{Appendix by Y. Canzani, D. Jakobson and L. Silberman\\ Lipschitz distance.  Applications to the 
study of diameter and Laplace eigenvalues.}\label{section:appendix}

In this section we shall prove tail estimates for a Lipschitz-type distance
$\rho$ defined below, and use those estimates to prove that the
\emph{diameter} and \emph{Laplace eigenvalue} functionals are measurable
with respect to the Gaussian measures defined in Section \ref{sec:Gaussian},
and to give tail estimates for them.  We maintain the hypothesis that all
metrics under consideration have the same associated volume form $\dv$, though
the results could be easily modified to remove this assumption.

\subsection{Lipschitz distance}\label{Lipschitz}

Here we study a (Lipschitz-type) distance $\rho$ related to the
distance used in \cite{BU} by Bando and Urakawa. It is defined by
\begin{equation}\label{dist:Linf}
\rho(g^0,g^1)=\sup_{x\in M}\ \sup_{0\neq\xi\in
T_xM}\left|\ln\frac{g^1(\xi,\xi)}{g^0(\xi,\xi)}\right|
\end{equation}
In other words, it is determined by taking the identity map on $M$ and
considering its Lipschitz constants between the two metrics.
Note that the fiber-wise constant is also the larger of the Lipschitz
constants of the map and its inverse: on the one hand, clearly for any
curve on $M$, its $g^1$-length is at most $\exp(\rho(g^0,g^1))$ times
its $g^0$ length, and conversely for the $(x,\xi)$ achieving the supremum,
taking $y$ near $x$ in the direction $\xi$ we see that the $g^0$ and $g^1$-
distances between $x,y$ are roughly distorted by the same factor (though
we do not know which is larger).

As in the case of $\Omega_2$, $\rho(g^1,g^0)$ depends only on $a_x$
where $g^1_x=k_x a_x \cdot g^0_x$ (action of $G_x$ on $\Hom(T_xM,T^*_xM)$
by composition; in the representation of metrics are positive-definite matrices
this is the congruence action $g^2_x = k_x a_x g^0_x a_x k_x^{-1}$).
In the our adapted frame, the diagonal
part of $g^1$ has entries $e^{2b_i(x)}$, where $\sum_i b_i(x)=0$
for every $x\in M$, and where the vector $b(x)=(b_1(x),\ldots,b_n(x))$
is defined by the formula \eqref{rand:radial}.  Specifically, for any
$x\in $M the second supremum in \eqref{dist:Linf} is equal to
\begin{equation}\label{local:sup}
2\sup_i |b_i(x)|
\end{equation}
The supremum is attained for $\xi=e_i$ (the $i$-th unit vector in $T_xM$).  
Accordingly,
\begin{equation}\label{dist:Linf:2}
\rho(g^0,g^1)=2\sup_{1\leq i\leq n}\sup_{x\in M} |b_i(x)|
\end{equation}

%%%%%%

\subsection{Tail estimate for $\rho$}
Now, $\rho>R$ iff $\sup_j\sup_{x\in M} |2b_i(x)|>R$.  Accordingly,
\begin{equation}\label{rho:ineq1}
\Prob\{\rho(g^0,g^1)>R\}\leq \Prob\{\sup_{x\in M}\sup_i |b_i(x)|>R/2\}.
\end{equation}
Recall that $\diag(b_1,\ldots,b_n)$ is given by projecting a random
vector on a particular hyperplane, which does not increase the maximum
norm.  It follows that 
$$
\Prob\{\rho(g^1,g^0)>R\}\leq \Prob\{\sup_{x\in M}\sup_j |a_j(x)|>R/2\}, 
$$
where $a_j$ are the components of an $\R^n$-valued Gaussian vector.
By symmetry we have for fixed $i$ that
$$
\Prob\{\sup_{x\in M} |a_i(x)|>u\} \leq 2\cdot\Prob\{\sup_{x\in M} a_i(x)>u\}.
$$
Taking the union bound we find that
\begin{equation}\label{eq:rhoineq}
\Prob\{\rho(g^0,g^1)>R\}\leq 2n\cdot\Prob\{\sup_{x\in M} a_1(x)>R/2\}.  
\end{equation}

We would like to estimate this probability as $R\to\infty$.  We will need
the covariance function for the scalar random field $a_1(x)$,
given by (see \eqref{eq:covariance})
$$
r_{a_1}(x,y)=\sum_{k=1}^\infty \beta_k^2\psi_k(x)\psi_k(y),
$$
where $\psi_k$ denote the $L^2$-normalized eigenfunctions of
$\Delta(g_0)$.

The following result now follows in a standard way from the Borell-TIS
theorem; it can be easily deduced from the calculations in \cite[\S
3]{CJW} and \cite[\S 2, (2.1.3)]{AT08}.  We denote by $\sigma^2$ the supremum
of the variance $r_{a_1}(x,x)$:
\begin{equation}\label{variance:sup}
\sigma^2:=\sigma(a_1)^2:=\sup_{x\in M} r_{a_1}(x,x).
\end{equation}
  
\begin{prop}\label{prop:excursion}
Let $\sigma(a_j)$ be as in \eqref{variance:sup}.  Then
$$
\lim_{R\to\infty} \frac{\ln\Prob\{\sup_{x\in M} a_1(x)>R/2\} }{R^2} =
\frac{-1}{8\sigma^2}.
$$
\end{prop}

Proposition \ref{prop:excursion} and \eqref{eq:rhoineq} imply the following 
\begin{cor}\label{tail:Lip}
Let $\sigma^2:=\sup_{x\in M} r_{a_1}(x,x)$.  Then 
\begin{equation}\label{rho:ineq4}
\lim_{R\to\infty} \frac{\ln\Prob\{\rho(g_0,g_1)>R\}}{R^2} \leq 
\frac{-1}{8\sigma^2}.
\end{equation}
\end{cor}

In the sequel, we shall need a slightly more precise estimate; it follows from the previous discussion 
and the estimates in \cite[\S 2, p. 50]{AT08}.  
\begin{proposition}\label{prob:Lip:precise}
There exists $\alpha>0$ such that for a fixed $\epsilon>0$ and for large enough $R$, we have 
$$
\Prob\{\rho(g_1,g_0)>R\}\leq 2n\exp\left(\frac{\alpha R}{2} - \frac{R^2}{8\sigma^2}\right).  
$$
\end{proposition}

%%%%%

\subsection{Diameter and eigenvalue functionals} 

In this section we use Corollary \ref{tail:Lip} to give estimates for the
diameter and Laplace eigenvalues of the random metric $g_1$.  

\begin{lemma}\label{lemma:diam:lambda}
Assume that $d\vol(g_0)=d\vol(g_1)$, and that in addition $\rho(g_0,g_1)<R$.  Then 
\begin{equation}\label{diameter:ratio}
e^{-R}\leq \frac{\diam(M,g_1)}{\diam(M,g_0)}\leq e^R.\end{equation}
and also 
\begin{equation}\label{eigenvalue:ratio} 
e^{-2R}\leq \frac{\lambda_k(\Delta(g_1))}{\lambda_k(\Delta(g_0))}\leq e^{2R}.
\end{equation}
\end{lemma}

\begin{proof}  
The definition \eqref{dist:Linf} implies that for any fixed path $\gamma:[0,1]\to M$, the ratio of its lengths 
with respect to the metrics $g_0$ and $g_1$ is satisfies 
$$
e^{-R}\leq \frac{L_{g_1}(\gamma)}{L_{g_0}(\gamma)}\leq e^R.
$$
Since 
$$
\diam(M,g)=\sup_{x,y\in M}\; \inf_{\gamma:\gamma(0)=x,\gamma(1)=y}\; L_g(\gamma), 
$$
the inequality \eqref{diameter:ratio} follows.  

To prove \eqref{eigenvalue:ratio}, we let 
$h\in H^1(M),h\not\equiv 0$ be a test function.  Then
$||h||_g^2:=\int_M h^2 \dv$ is independent 
of the metric, since the volume form $\dv$ is fixed.  The Rayleigh quotient of $h$ is equal to 
$$
\frac{\langle dh,dh\rangle_{g^{-1}}}{||h||_g^2}, 
$$
where $g^{-1}$ denotes the co-metric corresponding to $g$.  Since the Lipschitz distance is symmetric 
in its two arguments, we conclude that if $\rho(g_0,g_1)<R$, then $\rho(g_0^{-1},g_1^{-1})<R$ as well.  
It follows that 
\begin{equation}\label{Rayleigh:test}
e^{-2R}\leq \frac{\langle dh,dh\rangle_{g_0^{-1}}}{\langle dh,dh\rangle_{g_1^{-1}}}\leq e^{2R}. 
\end{equation}

By the min-max characterization of the eigenvalues (see e.g. \cite[\S 2]{BU}), 
$$
\lambda_k(\Delta(g))=\inf_{U\subset H^1(M):\; \dim U=k+1}\ \ 
\sup_{h\in U, \; h\not\equiv 0}\frac{||dh||_{g^{-1}}^2}{||h||_g^2}.
$$ 

The estimate \eqref{eigenvalue:ratio} now follows from \eqref{Rayleigh:test}.  

\end{proof}

We next establish some integrability results for the \emph{diameter}
functional $\diam(M,g_1)$.  They follow from Lemma \ref{lemma:diam:lambda}
and the stronger form of Corollary \ref{tail:Lip}.  

\begin{theorem}\label{diameter:integrable}
Let $h:\R^+\to\R^+$ be a monotonically increasing function such that
for some $\delta>0$ 
$$
h(e^y)=O\left(\exp\left[y^2(1/(8\sigma^2)-\delta)\right]\right).  
$$
Then $h(\diam(g_1))$ is integrable with respect to the probability measure
$d\omega(g_1)$ constructed in section \ref{sec:Gaussian}.  
\end{theorem}

In the proof we shall use Proposition \ref{prob:Lip:precise}.

%%%%%

\begin{proof}
Without loss of generality, assume that we have normalized $g_0$
so that $\diam(g_0)=1$.  We have shown in \eqref{diameter:ratio} that
if $\rho(g_0,g_1)<R$, then
$$\diam(g_1)\leq \diam(g_0)\cdot e^R=e^R,$$
so ($h$ being monotone) we have under the same assumption that
$$
h(\diam(g_1))<h(e^R).
$$
Since $h\geq 0$, the function $h(\diam(g_1))$ is integrable provided the sum 
$$
\sum_{k=N}^\infty h(e^k)\cdot\Prob\{g_1:k-1\leq\rho(g_1,g_0)\leq k\}
$$
converges.  By the hypotheses on $h$ and Corollary \ref{tail:Lip},
that sum is dominated by
$$
\begin{aligned}
\; & 2n\sum_{k=N}^\infty h(e^k)\exp\left(\frac{\alpha(k-1)}{2}-\frac{(k-1)^2}{8\sigma^2}\right)  \leq \\ 
\; & 2n\sum_{k=N}^\infty \exp\left[\frac{\alpha(k-1)}{2}+\left(\frac{k^2}{8\sigma^2}-\delta k^2\right)-
\frac{(k-1)^2}{8\sigma^2}\right]  
\end{aligned}
$$
Choosing $N$ large enough, we find that the last sum is dominated by 
$$
 2n\sum_{k=N}^\infty \exp\left[\frac{-\delta k^2}{2}\right], 
$$
and the last expression clearly converges.  
\end{proof}

\begin{remark}\label{distance:average}
The proof of Theorem \ref{diameter:integrable} can be easily modified  to establish analogous results for {\em averages} 
of the distance function.  For example, given $t>0$, consider the functional 
$$
E_t(g):=\int_M\int_M ({\rm dist}_g(x,y))^t\; \dv(x)\; \dv(y).  
$$
We leave the details to the reader.  
\end{remark}

Another corollary is the following 
\begin{theorem}\label{eigenvalue:integrable} 
Let $h : \R^+ \to \R^+$ be a monotonically increasing function such that for some $\delta>0$ 
$$
h(e^{2y})=O\left(\exp\left[y^2(1/(8\sigma^2)-\delta)\right]\right).
$$
Then $h(\lambda_k(\Delta(g_1)))$ is integrable with respect to the probability measure $d\omega(g_1)$ constructed in 
section \ref{sec:Gaussian}.  
\end{theorem}

\begin{proof}
The proof is similar to the proof of Theorem \ref{diameter:integrable}.  We let 
$\lambda_k(g_0) =: e^{2\beta_k} =: e^{2\beta}$.  

It follows from \eqref{eigenvalue:ratio} that if $\rho(g_0,g_1)<R$, then 
$\lambda_k(g_1)\leq \lambda_k(g_0)\cdot e^{2R}=e^{2(R+\beta)}$.  By monotonicity of the function $h$, we have 
$$
h(\lambda_k(g_1))<h(e^{2(R+\beta)}).  
$$
Since $h\geq 0$, the function $h(\lambda_k(g_1))$ is integrable provided the sum 
$$
\sum_{m=N}^\infty h(e^{2(m+\beta)})\cdot\Prob\{g_1:m-1\leq\rho(g_1,g_0)\leq m\}
$$
converges.  

By the assumptions on $h$ and Corollary \ref{tail:Lip}, that sum is dominated by
$$
\begin{aligned}
\; & 2n\sum_{m=N}^\infty h(e^{2(m+\beta)})\exp\left(\frac{\alpha(m-1)}{2}-\frac{(m-1)^2}{8\sigma^2}\right)  \leq \\ 
\; & 2n\sum_{m=N}^\infty \exp\left[\frac{\alpha(m-1)}{2}+\left(\frac{(m+\beta)^2}{8\sigma^2}-
\delta (m+\beta)^2\right)-\frac{(m-1)^2}{8\sigma^2}\right]  
\end{aligned}
$$

Choosing $N$ large enough, we find that the last sum is dominated by 
$$
2n \sum_{m=N}^\infty \exp\left[\frac{-\delta m^2}{2}\right], 
$$
and the last expression clearly converges.  

\end{proof}

\begin{remark}
Theorems \ref{diameter:integrable} and \ref{eigenvalue:integrable} prove integrability results about the 
diameter and eigenvalue functionals.   We plan to further study those and other functionals in future papers. 
\end{remark}

\subsection{Volume entropy functional}  
The \emph{volume entropy} functional $h_{vol}(g)$ of a metric $g$ was
defined by Manning in \cite{Manning} as the exponential growth rate of
volume in the universal cover (by showing that this growth rate is independent
of the point of reference). In other words, it was shown that for any point
$x$ in the universal cover $N$ of a compact Riemannian manifold $M$, the limit
\begin{equation}\label{volume:entropy:def}
h_{vol}=\lim_{s\to\infty} \frac{1}{s}\ln\vol (B(x,s)), 
\end{equation}
exists and is independent of the choice of $x$.  Here, volumes and distances
(and thus balls) in $N$ are with respect to the metric lifted from $M$.

We first prove the following counterpart of Lemma \ref{lemma:diam:lambda}.  
\begin{lemma}\label{lemma:vol:entropy}
Assume that $d\vol(g_0)=d\vol(g_1)$, and that in addition $\rho(g_0,g_1)<R$.
Then 
\begin{equation}\label{entropy:ratio}
e^{-R}\leq \frac{h_{vol}(M,g_1)}{h_{vol}(M,g_0)}\leq e^R.\end{equation}
\end{lemma}

\begin{proof}
By symmetry it is enough to prove the right-side inequality.  Since $\rho$
bounds the Lipschitz constant of the identity map (the argument above lifts
to the universal cover), we have (balls in $N$ with respect to the lifts
of the respective metrics)
\begin{equation}\label{eqn:inc-balls}
B_{g_1}(x,s)\subset B_{g_0}(x,e^R\cdot s).  
\end{equation}
By definition of $h_{vol}$, for any $\epsilon>0$ there exists $s_0>0$ such
that for every $s>s_0$, we have 
$$
\frac{1}{s}\ln\vol B_{g_0}(x,s)\leq h_{vol}(g_0)+\epsilon.  
$$

Combining the two claims, it follows that for $s>s_0$,
$$
\frac{1}{s}\ln\vol B_{g_1}(x,s)
   \leq e^R\frac{1}{s e^R} \ln\vol B_{g_0}(x,e^R s)
   \leq e^R(h_{vol}(g_0)+\epsilon)
$$
(we used here the assumption that $g_0,g_1$ have the same volume form,
so that the set-theoretic inclusion of balls implied an inequality on the
volumes; without the assumption the volume would be an additional factor from
the distortion of the volume form, but note that this factor would not
affect the the inequalit in the limit $s\to\infty$).

Letting $s\to\infty$ we obtain $h_{vol}(g_1) \leq e^R(h_{vol}(g_0)+\epsilon)$,
and letting $\epsilon\to 0$ we finally get
$$
h_{vol}(g_1) \leq e^R h_{vol}(g_0).
$$
\end{proof}

Lemma \ref{lemma:vol:entropy} now easily implies 
\begin{theorem}\label{entropy:integrable}
The conclusion of the Theorems \ref{diameter:integrable} remains true if the diameter functional 
is replaced by the volume entropy functional $h_{vol}$.  
\end{theorem}

%%%%%%%%%%%

\end{document}